\documentclass[11pt,oneside,reqno]{amsart}
\usepackage{palatino, mathpazo, amsmath, amsfonts, amssymb, mathrsfs,epic}
\usepackage[mathscr]{eucal}
\usepackage[all]{xy}
\usepackage{url}
\usepackage{alltt}

\usepackage{color, transparent}             
\usepackage[usenames, dvipsnames]{xcolor}	
\usepackage[colorlinks, breaklinks,
	bookmarks = false,
	pdfstartview = Fit,  
	pdfview = Fit,       
	linkcolor = Blue,
	urlcolor = Green,
	citecolor = Green,
	hyperfootnotes = false
	]{hyperref}
\usepackage{multirow}
\usepackage{stmaryrd}

 1
 1
 1

\newcommand{\ric}{\text{Ric}}

\newcommand{\bbC}{\mathbb{C}}
\newcommand{\bbN}{\mathbb{N}}
\newcommand{\bbQ}{\mathbb{Q}}
\newcommand{\bbZ}{\mathbb{Z}}
\newcommand{\bbR}{\mathbb{R}}
\newcommand{\bbH}{\mathbb{H}}
\newcommand{\bbA}{\mathbb{A}}

\newcommand{\bfH}{\mathbf{H}}
\newcommand{\bp}{\textbf{p}}

\DeclareMathOperator{\aut}{Aut}

\DeclareMathOperator{\p}{\partial}
\DeclareMathOperator{\pbar}{\overline{\partial}}

\numberwithin{equation}{section}

\theoremstyle{plain}
\newtheorem{thm}{Theorem}[section]
\newtheorem{prop}[thm]{Proposition}

\newtheorem{cor}[thm]{Corollary}
\newtheorem{lem}[thm]{Lemma}
\newtheorem{obser}[thm]{Observation}
\newtheorem{conj}[thm]{Conjecture}

\theoremstyle{definition} 
\newtheorem{defn}[thm]{Definition}

\theoremstyle{remark}
\newtheorem{remark}[thm]{Remark}

\theoremstyle{remark}
\newtheorem{example}[thm]{Example}

\title[]{A Hodge theoretic criterion for finite Weil--Petersson degenerations over a higher dimensional base}

\author[T.-J.~Lee]{Tsung-Ju~Lee}
\address{T.-J.~Lee: Department of Mathematics,
National Taiwan University, Taipei 10617, Taiwan}
\email{f97221051@ntu.edu.tw}

\begin{document}
\maketitle
\begin{abstract}
We give a Hodge-theoretic criterion for a Calabi--Yau variety to have finite Weil--Petersson distance on higher dimensional bases up to a set of codimension $\geq 2$. The main tool is variation of Hodge structures and variation of mixed Hodge structures. 

We also give a description on the codimension 2 locus for the moduli space of Calabi--Yau threefolds. We prove that the points lying on exactly one finite and one infinite divisor have infinite Weil--Petersson distance along angular slices. Finally, by giving a classification of the dominant term of the candidates of the Weil--Petersson potential, we prove that the points on the intersection of exact two infinite divisors have infinite distance measured by the metric induced from the dominant terms of the candidates of the Weil--Petersson potential.
\end{abstract}
\tableofcontents

\setcounter{section}{-1}

\section{Introduction}

\subsection{Motivations}
Viehweg proved that the coarse moduli space $S_h$ of polarized Calabi--Yau manifolds, with a fixed Hilbert polynomial $h$, exists and is quasi-projective \cite{V2010}. Together with Kawamata's result on the deformation invariance of canonical singularities \cite{K1999}, Viehweg's construction also leads to a quasi-projective moduli space of polarized Calabi--Yau varieties with at worst canonical singularities as a partial compactification.

Thanks to Yau's theorem on the Calabi conjecture, a canonical metric $g_{WP}$ on $S_h$, known as the \emph{Weil--Petersson metric}, is defined. This metric can be also defined by the corresponding variation of Hodge structures on $S_h$ and extended to its compactification (cf.~{\bf Section 1.3}). The classical Weil--Petersson metric on the moduli space of abelian varieties and polarized K3 surfaces are well-known to be complete. For a degeneration of Calabi--Yau manifolds $\mathfrak{X}/S$, $g_{WP}$ on the smooth locus is not necessarily complete, and it is thus natural to study its metric completion.  Indeed, in higher dimensions, C.-L.~Wang gave a Hodge-theoretic criterion for a singular Calabi--Yau to be at finite Weil--Petersson distance when it is smoothable along a \textit{one parameter family} \cite{W1997}. The moduli spaces, however, are of higher dimensions in general. He then posed the

\begin{conj} [Wang] \label{conj}
Let $\mathfrak{X}/S$ be an $n$-dimensional Calabi--Yau degeneration which is smooth outside a simple normal crossing divisor $\bigcup_i E_i$. Then $X:=\mathfrak{X}_s$, with $s \in \bigcup_i E_i$, has finite Weil--Petersson distance if and only if $N_iF^n_\infty(s)=0$ for all $i$ with $s \in E_i$. Here $F^n_\infty$ is the $n$-th piece of Schmid's limiting Hodge filtration and $N_i$ is the nilpotent part of monodromy around $E_i$.
\end{conj} 

Note that the situation in the conjecture can always be reached by a sequence of blow-ups on boundaries.

Wang also showed that if the central fiber in a one parameter degeneration $ \mathfrak{X}\to \Delta$ is at finite Weil--Petersson distance, then after running a minimal model program, the central fiber $\mathfrak{X}_0'$ of the output $\mathfrak{X}'\to\Delta$ has at worst canonical singularities \cite{W1997}. This picture connects the finite distance property with the geometry of the central fiber when the base is one-dimensional. We may therefore regard Viehweg's moduli $S_h$ as \emph{charts} on the metric completion of Calabi--Yau moduli, and these charts form a \emph{covering} when $h$ varies. The goal of this note is to understand this picture over higher dimensional bases by way of studying {\bf Conjecture \ref{conj}}.

For higher dimensional case, the asymptotic behavior of a degeneration of Hodge structures were studied in \cite{K1985}, \cite{CKS1986} by means of the multivariable $SL_2$-orbit theorem. The Weil--Petersson potential and the metric are thus controlled in the \emph{sectors} (for more notations and details, see \cite{K1985}). For the curves lie in a sector, we have an estimate on the distance as in the one parameter case. {\bf Conjecture \ref{conj}} will then follow if one can show that the geodesic always lies in a sector. This point of view, which was first suggested to me by Wang, is closed related to the \emph{algebraicity} or \emph{regularity} of the geodesics towards the degenerate boundary point. However, since this is not known yet, we will not take this approach in this paper.

A differential geometric approach to connect the finiteness of the Weil--Petersson distance with Calabi--Yau variety with canonical singularities was recently developed by Takayama \cite{S2015} and Tosatti \cite{T2015}, after the fundamental work of Donaldson and Sun \cite{DS2014}. By developing the theory of Gromov--Hausdorff convergence to construct the limiting variety, the finiteness of Weil--Petersson distance along a holomorphic curve was shown to be equivalent to the uniform boundedness of diameters in the family. The case when the family is over a higher dimensional base is also not known in this approach; namely, the finiteness of Weil--Petersson distance along a real curve $\gamma$ should imply the diameter boundedness of the induced family along $\gamma$. On the other hand, the validity of {\bf Conjecture \ref{conj}} will imply this diameter boundedness statement (cf.~{\it Remark} \ref{diam}). This again gives us strong motivation to study {\bf Conjecture \ref{conj}}.

\subsection{Statements of the main results}
In this paper, as the first step to this project, we give an analogous Hodge-theoretic criterion for a Calabi--Yau variety to have finite Weil--Petersson distance. To be precise, we prove that:

\begin{thm}[= {\bf Theorem \ref{Thm1-1}} + {\bf Theorem \ref{Thm1-2}}]\label{thm02}
{\textit{\textbf{Conjecture \ref{conj}}}} holds for the degeneration of Calabi--Yau $n$-folds, up to a set of codimension $\geq 2$ in the base $S$.
\end{thm}

The idea of the proof is using variation of mixed Hodge structures on codimension one boundary points. This puts a strong constrain on the Weil--Petersson potential as explained in {\bf Section 2.2} and hence the metric is controlled. 

The case of intersections of \emph{finite divisors} (i.e.~$E_i$ with $N_i F^n_\infty = 0$, see {\bf Definition \ref{d:inf-div}}) is a consequence of the one-parameter case. The locus of intersections of more than one \emph{infinite divisors} (i.e.~$N_i F^n_\infty \ne 0$) are subtle. It is not easy to pick out the dominant term of the metric in this case. Nevertheless, the case of Calabi--Yau threefolds turns out to be easier to handle because of the symmetry of mixed Hodge diamonds. In the case of a two parameter family of Calabi--Yau threefolds, we can give a complete description of the dominant term of Weil--Petersson potentials, which is called a candidate of the Weil--Petersson potential (cf.~{\bf Definition \ref{cWP}}).

The metric is computed on the covering $\bbH^2\times(\Delta)^{r-2}\to (\Delta^*)^2\times(\Delta)^{r-2}$. We define the \emph{angular slice} to be the set (where $c_j \in \mathbb{R}$ are constants)
\[
\{(z_1,z_2)\in\bbH^2:\Re{z_j}=c_j\}\times(\Delta)^{r-2}.
\] 
Toward a proof of {\textbf{Conjecture \ref{conj}}} in this case, we have: 
\begin{thm}[=\ {\bf Theorem \ref{main3}}]
Let $\mathfrak{X}/S$ be a degeneration of Calabi--Yau threefolds as stated in {\bf Conjecture \ref{conj}}. Suppose $s\in S$ lies on exactly two boundary divisors, say $s\in E_1\cap E_2$, with $E_1$ infinite and $E_2$ finite. Then $s$ has infinite Weil--Petersson distance {\it along the angular slices}.
\end{thm}

The case that $s$ lies in the intersection of two infinite divisors is more complicated. As a beginning step toward this situation, we show that:

\begin{prop}[=\ {\bf Proposition \ref{prop}}]
In the case of Calabi--Yau threefolds, suppose $s\in S$ lies exact on two boundary divisors, say $s\in E_1\cap E_2$, with $E_i$'s being infinite. Then it has infinite distance measured by the dominant term of the candidates of the Weil--Petersson potentials. 
\end{prop}

(See {\bf Corollary \ref{cor2}} for some cases that the ``dominant'' assumption can be removed.) It should be pointed out that there are many hidden constrains on a Weil--Petersson potential, which reflect some non-trivial relations on the polarization $\tilde{Q}$. Classification of such potentials is also an important way to understand the (mixed) Hodge theory on the moduli. 

\begin{remark}
The Weil--Petersson metric is defined for degenerations of $n$-dimensional compact K\"{a}hler manifolds with $h^{n,0}=1$ (cf.~{\bf Section} 1.3). All the results stated above hold in this more general setting. 
\end{remark}
\subsection{Acknowledgement} 
I am grateful to Professor C.-L.~Wang for introducing this problem to me and for sharing with me his viewpoints. I am grateful to him and to Professor H.-W.~Lin for their steady encouragement. I am also grateful to Professors Z.~Lu and J.-D.~Yu, and Dr.~T.-C.~Yang for discussions. Results in this paper had been presented in the conference \emph{Calabi--Yau Geometry and Mirror Symmetry} in 2014 at National Taiwan University. I am grateful to Professor S.-T.~Yau for the invitation and for his interest in this work.

\section{Preliminaries}
In this section, we recall some preliminaries and fix some notations.
\subsection{Variation of Hodge structures} Let $(X,\omega)$ be an $n$-dimensional compact K\"{a}hler manifold. $V_\bbR:=H^m(X,\bbR)_{prim}$. We have the primitive Hodge decomposition on $V:=V_\bbR\otimes\bbC$: $V=\bigoplus_{p+q=m} P^{p,q}$ with $P^{p,q}=\overline{P^{q,p}}$, or equivalently, a Hodge filtration $F^\bullet$ on $V$: $V=F^0\supset F^1\supset\cdots\supset F^m$ with $F^p:=\bigoplus_{i\geq p} P^{i,m-i}$. On the piece $m\leq n$, we have the \textit{Hodge-Riemann bilinear form}:
\[
Q(u,v):=(-1)^{m(m-1)/2}\int_X u\wedge v\wedge \omega^{n-m},
\]
which polarizes $V$ in the sense that {\bf (a)} $Q(F^{p},F^{m+1-p})=0$ for all $p$ and {\bf (b)} $Q(Cv,\overline{v})>0$. Here $C$ is the \textit{Weil operator}.

Let $\mathfrak{X}\to S$ be a family of $n$-dimensional K\"{a}hler manifolds with a fixed K\"{a}hler class $[\omega]$. We pick a reference fiber $\mathfrak{X}_\circ$. All the primitive cohomology groups can be identified with $V$($=H^m(\mathfrak{X}_\circ,\bbC)_{prim}$). The corresponding data $(V,F^\bullet_s)$ above varies and forms a \textit{variation of Hodge structures}. It satisfies the Griffiths transversality. 

Let $D$ be the period domain which classifies all the Hodge filtration $F^\bullet$ on $V$ polarized by $Q$. $D$ can be realized as $G/K$, where $G=\aut(V_\bbR,Q)$ and $K$ the stabilizer of a reference point. $D$ has a natural compactification $\check{D}$, which consists of the filtrations only satisfying condition {\bf (a)}. It contains $D$ as an open dense subvariety. The family $\mathfrak{X}/S$ gives the \textit{period map} $\phi: S\to \Gamma\backslash D$. Here $\Gamma$ is the representation of $\pi_1(S)$ inside $G$. Finally, note that the period map is \emph{locally liftable} to $D$.

\subsection{Schmid's mixed Hodge theory} Let $\Delta:=\{z\in\bbC: |z|<1\}$ and $\Delta^*:=\Delta-\{0\}$. Put $S=\Delta^r$ and $S^\circ=(\Delta^*)^k\times\Delta^{r-k}$. Let $\phi:S^\circ\to \Gamma \backslash D$ be a polarized variation of Hodge structures over $S^\circ$ polarized by $Q$. Let $T$ be the universal cover of $S^\circ$. We have the following commutative diagram
\begin{displaymath}
    \xymatrix{
        T \ar[r]^\Phi \ar[d] & D \ar[d] \\
        S^\circ \ar[r]^{\phi}       & \Gamma\backslash D }
\end{displaymath}
Indeed, $T=\bbH^k\times\Delta^{r-k}$. We will use $z_i$ for the coordinates on $T$ and $t_i$ for the coordinates on $S^\circ$. They are related by $t_s=e^{2\pi\sqrt{-1}z_s}$, $1\leq i\leq k$, and $z_s=t_s$, $k+1\leq s\leq r$. Let $N_i$ be the monodromy operator around $t_i=0$, $1\leq i\leq k$. Define 
\[
A(z):=\exp(-\sum_{i=1}^k z_iN_i)\Phi(z).
\]
Then $A(z)$ is invariant under the action $z_i\mapsto z_i+1$. It descends to a holomorphic function on $S^\circ$, which is denoted by $a(t)$. By Schmid's nilpotent orbit theorem \cite{S1973}, the function $a(t)$ extends over $S$ \emph{holomorphically} to $\check{D}$. The extended values $F_\infty(t):=a(t)$ for $t\in S-S^\circ$ are called the \emph{limiting Hodge filtration}. 

In the case $r=k=1$, the monodromy operator $N$ defines an increasing filtration, called the monodromy weight filtration, $0=W_{-1}\subset W_0\subset\cdots\subset W_{2n}=V$ on $V$ so that $N(W_l)\subset W_{l-2}$ and there are isomorphisms $N^s :Gr^W_{n+s}\to Gr^W_{n-s}$, where $Gr^W_s:=W_s/W_{s-1}$. The triple $(V, W_\bullet, F_\infty(0))$ forms a mixed Hodge structure and $N$ is a morphism of weight $(-1,-1)$ \cite{S1973}. Moreover, for $s\geq 0$, the primitive part $P^W_{n+s}:=\ker N^{s+1}\subset Gr^W_{n+s}$ is polarized by $Q(\cdot,N^s\overline{\cdot})$.

For multivariable case, Cattani and Kaplan showed that the monodromy weight filtration behaves well in the following sense \cite{CK1982}. Let $\sigma$ be the cone $\{\sum_{i=1}^k a_iN_i: a_i>0\}$. Any $N\in \sigma$ induces \emph{the same} monodromy weight filtration on $V$, denoted by $W(\sigma)$. Furthermore, if $\tau_1$, $\tau_2$ are faces of $\sigma$ and $N\in\tau_1$ is not contained in the closure of $\tau_2$, then $W(\tau_1)$ is the monodromy weight filtration of $N$ relative to $W(\tau_2)$. We also remark that $N_i$ commutes with $N_j$ for all $i,j$. 

\subsection{The Weil--Petersson metric} We shall only consider the Calabi--Yau case. For a $n$-dimensional compact K\"{a}hler Calabi--Yau manifold $(X,\omega)$, by Bogomolov-Tian-Todorov unobstructedness theorem \cite{T1987}, \cite{T1989}, the local universal deformation space is smooth. Let $\mathfrak{X}/S$ be the maximal subfamily with a fixed K\"{a}her class $[\omega]$. The Kodaira-Spencer map $\rho: T_s(S)\to H^1(\mathfrak{X}_s,T_{\mathfrak{X}_s})$ is injective. Since each $\mathfrak{X}_s$ is a Calabi--Yau manifold, by Yau's theorem, we can choose a unique Ricci flat metric in the K\"{a}hler class, say $g(s)$, up to a volume normalization. For $u,v\in T_s(S)$, we define
\[
G(u,v):=\int_{\mathfrak{X}_s} \langle \rho(u),\rho(v)\rangle_{g(s)}.
\]
The resulting metric is called the \textit{Weil--Petersson} metric. Using the global non-vanishing $n$-form $\Omega(s)$, one can easily compute
\[
G(u,v)=\frac{Q(C\iota(u)\Omega,\overline{\iota(v)\Omega})}{Q(C\Omega,\overline{\Omega})},
\]
where $\iota(u)$ is the interior product. It induces an isomorphism 
\[
H^1(\mathfrak{X}_s,T_{\mathfrak{X}_s})\cong \hom(H^{n,0},H^{n-1,1}),~u\mapsto \iota(u).
\]
For convenience, put $\tilde{Q}:=\sqrt{-1}^n Q$. It was proved that $\tilde{Q}$ is the K\"{a}hler potential of the Weil--Petersson metric \cite{T1987}. In fact, the metric two form is given by
\begin{equation}\label{WPpotential}
\omega_{WP}=\frac{\sqrt{-1}}{2}\ric_{\tilde{Q}}(H^{n,0})=-\frac{\sqrt{-1}}{2}\p\pbar\log\tilde{Q}.
\end{equation}
And the metric tensor is given by 
\begin{equation}\label{WPmetric}
g_{WP}=-\sum_{i,j} \p_i\p_{\overline{j}} (\log\tilde{Q}) (dz_i\otimes d\overline{z}_j+d\overline{z}_j\otimes dz_i)
\end{equation}
where $\p_i:=\p/\p z_i~\text{and}~\p_{\overline{j}}:=\p/\p {\bar z}_j$. Given a variation of Hodge structures polarized by $Q$ on $S$ with $H^{n,0}\cong\bbC$, we may use  \eqref{WPpotential} as our definition of the Weil--Petersson metric on $S$ although it is only semi-positive \cite{G1984}.

\subsection{Variations of mixed Hodge structures}
A degeneration $\mathfrak{X}/S$ of compact K\"{a}hler manifolds is a family of compact K\"{a}hler varieties whose general fibers are smooth. For convenience, we shall say that ``$\mathfrak{X}/S$ is a degeneration".

\begin{defn}
Let $S$ be a complex manifold. A variation of mixed Hodge structures over $S$ consists of the following data:
\begin{itemize}
\item[(1)] a local system $V$ of finitely generated groups on $S$.
\item[(2)] a finite decreasing filtration $\{\mathcal{F}^p\}$ of the holomorphic vector bundle $\mathcal{V}:=V\otimes\mathcal{O}_S$ by holomorphic subbundles.
\item[(3)] a finite increasing filtration $\{\mathcal{W}_m\}$ of the local system $V_\bbQ:=V\otimes_\bbZ\bbQ$ by local subsystems.
\end{itemize}
subject to the following conditions:
\begin{itemize}
\item[(a)] For each $s\in S$, these data forms a mixed Hodge structure when restricted to $s$.
\item[(b)] The connection $\nabla:\mathcal{V}\to \mathcal{V}\otimes\Omega^1_S$ determined by the local system $V_\bbC$ satisfies the Griffiths' transversality.
\end{itemize}
\end{defn}
Let $\mathfrak{X}/S$ be a degeneration with a fixed K\"{a}hler class $[\omega]$ and $S^\circ$ be its smooth locus. After some birational modifications, we assume that $S-S^\circ=\bigcup_i E_i$ is a simple normal crossing divisor. A point $s\in S$ is called a \textit{$k$-boundary point} it lies exactly on $k$ smooth divisors. For $x\in E_i-\bigcup_{j\neq i} E_j$, we have a limiting Hodge filtration $F_\infty(x)$ and the monodromy weight filtration determined by $N_i$. When $x$ varies (away from $\bigcup_{j\neq i}E_j$), these data fit together and form a \emph{variation of mixed Hodge structures}, \cite{S1973}, \cite{CK1982} and \cite{K2010}. See also \cite{K1981}.

\section{A Finite Distance Criterion along Boundary Divisors}

To prove the conjecture, we need to compute the Weil--Petersson distance of all $k$-boundary points. In this section, we consider the case $k=1$ and solve the conjecture completely.

\subsection{Review of the one parameter case} 
We will review the result for $r=k=1$. The finite distance property is equivalent to a  Hodge-theoretic property. To be precise, we have
\begin{thm}[=\ {\bf Theorem 1.1} \cite{W1997}]
Let $X$ be a Calabi--Yau variety which admits a (one parameter) smoothing to Calabi--Yau manifolds. Then $X$ has finite Weil--Petersson distance if and only if $NF^n_\infty=0$.
\end{thm} 

One direction of the {\bf Conjecture \ref{conj}} is just an application of this.

\begin{thm}\label{Thm1-1}
Let $\mathfrak{X}/S$ be a Calabi--Yau degeneration, which is smooth outside a simple normal crossing divisor $\bigcup_i E_i$. Then $X:=\mathfrak{X}_s$ has finite Weil--Petersson distance if $N_iF^n_\infty(s)=0$ for all $i$ with $s\in E_i$.
\end{thm}

\begin{proof}
Let $S^\circ:=S-\bigcup E_i$. Choose a \textit{holomorphic} curve $C$ passing through $s$ and meeting $S^\circ$. Pulling back everything along $C$, we reduce to the one parameter case. $X$ has finite Weil--Petersson distance along $C$, and hence along $S$ as desired.
\end{proof}

After running a minimal model program, we can replace $X$ by a minimal model \cite{W2002}. The required MMP now exists in our setting \cite{F2011}. Hence
\begin{thm}[=\ {\bf Proposition 1.2} \cite{W2002}]
Let $\mathfrak{X}/\Delta$ be a one dimensional degeneration as above. If $\mathfrak{X}_0$ has  finite Weil--Petersson distance, then after running a minimal model program, $\mathfrak{X}_0$ can be replaced by a Calabi--Yau variety with at worst canonical singularities.
\end{thm}

\subsection{Generalization to higher dimensional bases}

We will give a partial answer to the opposite direction in {\bf Conjecture \ref{conj}} in this subsection. Let $\pi:\mathfrak{X}\to S$ be such a degeneration on the moduli space. Let $S^\circ$ be the $\pi$-smooth locus and $S-S^\circ=\bigcup_i E_i$. After some blow-ups, we may assume $\sum E_i$ is a simple normal crossing divisor. Put $E_i^\circ=E_i-\bigcup_{j\neq i}E_j$.

Recall that the limiting filtration $F_\infty(s)$ is a function on $s\in E$. Suppose there exists an $s\in E_i^\circ$ so that $N_i F^n_\infty(s)\neq 0$. Then $N_i F^n_\infty(t)\neq 0$ for all $t\in E_i$ near $s$. Conversely, if $N_i F^n_\infty(s)=0$ for some $s\in E^\circ_i$, then $N_i F^n_\infty(t)= 0$ \textit{for all} $t$ near $s$. This follows from the fact that the limiting Hodge filtrations fit together and form \textit{a variation of mixed Hodge structures} on $E^\circ_i$. Since $E^\circ_i$ is connected, we must have $N_i F^n_\infty(s)=0$ \textit{for all} $s\in E^\circ_i$ or $N_i F^n_\infty(s)\neq 0$ \textit{for all} $s\in E^\circ_i$. This leads to the following definition. 

\begin{defn} \label{d:inf-div}
An irreducible boundary divisor $E_i$ is called an \textit{infinite divisor} (a \textit{finite divisor}, respectively), if there exists $s\in E_i^\circ$ such that $N_iF_\infty^n(s)\neq 0$ ($N_iF_\infty^n(s)= 0$, respectively).
\end{defn}

\begin{remark}
If $s\in E_i^\circ$ satisfies $N_iF^n_\infty(s)=0$, then it has finite distance along a holomorphic curve. Restricting the family to the curve and running the minimal model programing for the total family, the central fiber can be replaced by a Calabi--Yau variety with at worst canonical singularities. The property that $N_iF^n_\infty(t)=0$ for all $t$ near $s$ reflects the fact that the canonical singularities are deformation invariant \cite{K1999}.
\end{remark}

In this section, we prove the following:
\begin{thm}\label{Thm1-2}
If $s\in E_1$ with $E_1$ being an infinite divisor, and $s \not\in E_j$ for any $j \ne i$, i.e~$s$ is a $1$-boundary point,  then $\mathfrak{X}_s$ has infinite Weil--Petersson distance. 
\end{thm}

It suffices to prove this in local case. Let
\[
\Omega(z):=[\exp(z_1N_1)A(z)]^n.
\]
Here $[*]^n$ means the projection to the $n$-th flag. It represents a holomorphic top form on $\mathfrak{X}_{z}$. $A(z)$ has a holomorphic extension $a(t)$ over $S$. Locally near $t=0$, 
\begin{align}\label{t-expansion}
[a(t)]^n=a_0+\sum_I a_I t^I=a_0+\sum_{|I|\geq 1,~ i_1=0} a_It^I+\sum_{|I|\geq 1,~i_1\geq 1} a_It^I.
\end{align}
Denote by $\bbA$ the pull-back of $[a(t)]^n$ to $T$ (via $t_1=\exp(2\pi\sqrt{-1}z_1)$ and $t_s=z_s$ for $s\geq 2$), we have
\[
\bbA(z)=g(z_2,\cdots,z_r)+h(z_1,z_2,\cdots,z_r).
\]
$e^{z_1N_1}\bbA=\Omega$ on $S^\circ$. We can compute the Weil--Petersson metric using $e^{z_1N_1}\bbA$. Observe that $e^{z_1N_1}$ is indeed a polynomial and the series $h$ involves the factor $e^{2\pi\sqrt{-1}z_1}$. Let $z_s=x_s+\sqrt{-1}y_s$. For any polynomial $p(y_1)$, $p(y_1)h$ will exponentially decade as $y_1\to\infty$, uniformly in $z_s$ for $s\geq 2$ and all the partial derivatives of $p(y_1)h$ will do. For convenience, we denote by $\bfH$ the class of functions with the property stated above or elements in it. Then, since $N_1$ is nilpotent, 
\begin{align*}
\tilde{Q}(\Omega,\overline{\Omega})&=\tilde{Q}(e^{z_1N_1}(g+h),\overline{e^{z_1N_1}(g+h)})\\
&=\tilde{Q}(e^{2\sqrt{-1}y_1N_1}(g+h),\overline{g+h})\\
&=\tilde{Q}(e^{2\sqrt{-1}y_1N_1}g,\overline{g})+\bfH.
\end{align*}
The first term is a polynomial in $y_1$. Write
\[
\tilde{Q}(e^{2\sqrt{-1}y_1N_1}g,\overline{g})=\sum_{l=0}^d s_l(z_2,\cdots,z_r)y_1^l=:p(y_1).
\]
$s_l$'s are \textit{analytic} function in $z_2,\cdots, z_r$. Here is the key observation. 
\begin{obser}\label{obv}
If $E_1$ is a finite divisor, then $p(y_1)$ is a degree (in $y_1$) zero polynomial. If $E_1$ is an infinite divisor, then $p(y_1)$ is a degree $d$ polynomial with $s_d(z_2,\cdots,z_r)\neq 0$ for all $z_i$ near $0$.
\end{obser}
This is a consequence of the variation of mixed Hodge structures. It gives a strong constrain on the degree of the polynomials \textit{in a neighborhood} so that the Weil--Petersson metric is uniform in this neighborhood. For instance, the potential function can not take the form $y_1+y_2y_1^2$; namely, the higher degree $y_1$-term occurs with a coefficient going to zero when $y_2\to 0$. Note that the number $d$ indicates which graded piece in the monodromy weight filtration $a_0$ sits in. The metric two form
\[
\omega_{WP}=-\frac{\sqrt{-1}}{2}\p\pbar\log\tilde{Q}(\Omega,\overline{\Omega})=-\frac{\sqrt{-1}}{2}\sum_{i,j}\p_i\p_{\bar j}\log\tilde{Q}(\Omega,\overline{\Omega})dz_i\wedge d\overline{z}_j.
\]
Suppose now $E_1$ is an infinite divisor, i.e., $d\geq 1$. Then we have 
\begin{lem}
\[
-\frac{\p}{\p z_1}\frac{\p}{\p \overline{z}_1} \log\tilde{Q}\sim \frac{d}{y_1^2}
\]
as $y_1$ large.
\end{lem}
\begin{proof}
The twice differentiation is just the Laplace operator up to a constant. In this case, $\tilde{Q}=p(y_1)+\bfH$. One computes
\begin{align*}
-\frac{\p^2}{\p x_1^2}\log(p(y_1)+\bfH)=-\frac{\p}{\p x_1}\frac{\bfH}{p(y_1)+\bfH}=-\frac{\bfH}{(p(y_1)+\bfH)^2}.
\end{align*}
and
\begin{align*}
-\frac{\p^2}{\p y_1^2}&\log(p(y_1)+\bfH)=-\frac{\p}{\p y_1}\frac{p'(y_1)+\bfH}{p(y_1)+\bfH}\\
&=-\frac{(p''(y_1)+\bfH)(p(y_1)+\bfH)-(p'(y_1)+\bfH)^2}{(p(y_1)+\bfH)^2}\\
&\sim -\frac{d(d-1)s_d^2y_1^{2d-2}-d^2s_d^2y_1^{2d-2}}{s_d^2 y_1^{2d}}=\frac{d}{y_1^2}
\end{align*}
Note that ``$\sim$" holds since $s_d$ is bounded away from $0$. Summing together, we obtain the result.
\end{proof}

\begin{lem}\label{1j}
\[
-\frac{\p}{\p z_1}\frac{\p}{\p \overline{z}_j} \log\tilde{Q}
\]
is dominated by $C_j/y_1^2$ as $y_1$ large, where $C_j$ is a constant depending on $j$.
\end{lem}
\begin{proof}
One computes
\[
-\frac{\p}{\p x_j}\frac{\p}{\p x_1}\log(p(y_1)+\bfH)=-\frac{\p}{\p x_j}\frac{\bfH}{p(y_1)+\bfH}.
\]
The term $\p p(y_1)/\p x_j$ is at most a polynomial in $y_1$ of degree $d$. Thus
\[
-\frac{\p}{\p x_j}\frac{\bfH}{p(y_1)+\bfH}\sim\frac{\bfH}{(p(y_1)+\bfH)^2}\frac{\p p(y_1)}{\p x_j}=\bfH.
\]
Also,
\begin{align}\label{metric1j}
-&\frac{\p}{\p y_j}\frac{\p}{\p y_1}\log(p(y_1)+\bfH)=-\frac{\p}{\p y_j}\frac{p'(y_1)+\bfH}{p(y_1)+\bfH}\\
&\frac{\p}{\p y_j}(p'(y_1)+\bfH)=d\cdot \frac{\p s_d}{\p y_j}y_1^{d-1}+\cdots\notag\\
&\frac{\p}{\p y_j}(p(y_1)+\bfH)=\frac{\p s_d}{\p y_j}y_1^{d}+\cdots\notag
\end{align}
Therefore,
\begin{align*}
(p(y_1)+\bfH)\frac{\p}{\p y_j}(p'(y_1)+\bfH)-(p'(y_1)+\bfH)\frac{\p}{\p y_j}(p(y_1)+\bfH)
\end{align*}
is asymptotic to a polynomial in $y_1$ with degree at most $2d-2$. (\ref{metric1j}) is dominated by $C_j/y_1^2$ with a constant $C_j$.

Finally, for the mixed term,
\[
\frac{\p}{\p x_j}\frac{\p}{\p y_1}\log(p(y_1)+\bfH)=\frac{\p}{\p x_j}\frac{p'(y_1)+\bfH}{p(y_1)+\bfH}.
\]
This can be dominated by $C_j/y_1^2$ as in the previous case. We obtain the result by summing up the above discussions.
\end{proof}

\begin{proof}[Proof of Theorem \ref{Thm1-2}]
Following (\ref{WPmetric}) and {\bf Lemma} \ref{1j}, 
\[
-|g_{1\overline{j}}||dz_1\otimes d\overline{z}_j|\geq  -\frac{C_j}{y_1^2}\left|dz_1\otimes d\overline{z}_j\right|\geq -\left(\frac{\epsilon^2}{y_1^2}dz_1\otimes d\overline{z}_1+\frac{C_j^2}{\epsilon^2 y_1^2}dz_j\otimes d\overline{z}_j\right)/2.
\]
Then
\[
(\ref{WPmetric})\geq \left(\frac{d}{y_1^2}-\frac{\epsilon^2(r-1)}{y_1^2}\right)dz_1\otimes d\overline{z}_1+\sum_{i,j=2}^r g_{i\overline{j}} dz_i\otimes d\overline{z}_j-\frac{1}{2\epsilon^2y_1^2}\sum_{j=2}^r C_j^2dz_j\otimes d\overline{z}_j.
\]
We pick $\epsilon$ small enough so that $A:=d-\epsilon^2(r-1)$ is positive. When we integrate over any curve, the second term will be non-negative and the third term will be finite. Now let $p\in S^\circ$ and $\gamma$ be any real curve connecting $s$ with $p$.
\[
\int_\gamma ds\geq \int_{c}^\infty \frac{\sqrt{A}}{y_1}dy_1\pm(\text{finite terms})=\infty.
\]
\end{proof}

\begin{remark} Define the degree function on $E_i$ by
\[
d_i(s):=\max\{k\in\bbN:N_i^kF_\infty^n(s)\neq 0\},~~\text{if the set is non-empty}
\]
and we set $d_i(s)=0$ if $N_iF^n_\infty(s)=0$. $d_i(s)$ is lower semi-continuous. It also follows  from variation of mixed Hodge structures that this function is constant on $E_i-\cup_{j\neq i} E_j$. Thus for a finite divisor $E_i$, we have $d_i(s)\equiv 0$. For an infinite divisor $E_i$, $d_i(s)$ may jump down on further intersections.

For Calabi--Yau threefolds, we have a nice control on degree functions. Indeed, using notations in \eqref{t-expansion}, $N_i^{d_i(0)+1}a_I=0$ for any $a_I$. This fact follows from the structure of mixed Hodge diamonds. This guarantees that $d_i(0)\geq d_i(s)$ for $s$ near $0$. It follows that the degree functions are locally constant in threefold case.
\end{remark}

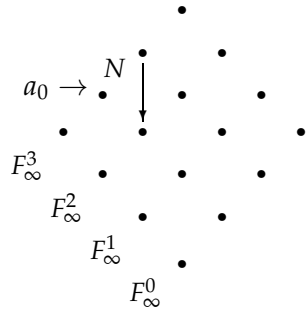
\begin{picture}(0,110)
	\put(80,100){\circle*{3}}
	\put(65,84){\circle*{3}} \put(95,84){\circle*{3}}
	\put(50,68){\circle*{3}} \put(80,68){\circle*{3}} \put(110,68){\circle*{3}} 
	\put(35,54){\circle*{3}} \put(65,54){\circle*{3}} \put(95,54){\circle*{3}} \put(125,54){\circle*{3}}
	\put(50,38){\circle*{3}} \put(80,38){\circle*{3}} \put(110,38){\circle*{3}} 
	\put(65,22){\circle*{3}} \put(95,22){\circle*{3}}
	\put(80,4){\circle*{3}}
			
	\put(65,80){\vector(0,-2){22}}
	\put(50,75){\small $N$}
	\put(20,68){$a_0 \to $}
	
	\put(15,38){\small $F^3_\infty$}
	\put(30,22){\small $F^2_\infty$}
	\put(45,6){\small $F^1_\infty$}
	\put(60,-10){\small $F^0_\infty$}
		
	\put(160,88){If $a_0$ degenerates to the indicated place,}
	\put(160,74){there is no element $a_I$ such that}
	\put(160,60){$N^ka_I\neq 0$ for $k\geq 2$.}
	\put(160,46){Otherwise, the first two rows should}
	\put(160,32){be non-zero. This can happen only if}
	\put(160,18){$a_0$ degenerates to the first two rows.}
	\put(160,4){The other cases can be checked similarly.}	
\end{picture}
\quad \\

\section{Two Parameter Family of Calabi--Yau Threefolds}
To describe the Weil--Petersson metric in the case $k\geq 2$, let us examine the potential function $\tilde{Q}$ more carefully. We assume at this moment that $r=k=2$ and the divisors $E_1$, $E_2$ are all infinite divisors. The holomorphic function $a(t)$ can be expressed locally as
\begin{align}
a(t)=&a_0+\sum_{|I|\geq 1,~i_2=0} a_It^I+\sum_{|I|\geq 1,~i_1=0} a_It^I+(\text{remaining terms})\\
&=:a_0+f_1(t_1)+f_2(t_2)+h(t)\label{expansion}.
\end{align}
For convenience, write $A_1:=e^{2\sqrt{-1}y_1N_1}$, $A_2=e^{-2\sqrt{-1}y_2N_2}$ and put $d_i:=\max\{l\in\bbN:N_i^{l}a_0\neq 0\}$. We have
\begin{align}\label{asym}
\tilde{Q}(\Omega,\overline{\Omega})&=\tilde{Q}(e^{2\sqrt{-1}y_1N_1}\bbA,e^{-2\sqrt{-1}y_2N_2}\overline{\bbA})\\
&=\tilde{Q}(A_1a_0,A_2\overline{a}_0)+\tilde{Q}(A_1a_0,A_2\overline{f}_1)+\tilde{Q}(A_1a_0,A_2\overline{f}_2)\notag\\
&\quad+\tilde{Q}(A_1f_1,A_2\overline{a}_0)+\tilde{Q}(A_1f_1,A_2\overline{f}_1)+\tilde{Q}(A_1f_1,A_2\overline{f}_2)\notag\\
&\quad+\tilde{Q}(A_1f_2,A_2\overline{a}_0)+\tilde{Q}(A_1f_2,A_2\overline{f}_1)+\tilde{Q}(A_1f_2,A_2\overline{f}_2)+\bfH_{12}\notag.
\end{align}
Here, of course, $\bfH_{12}$ is the function class with the property that any partial derivative decays exponentially as $y_1$, $y_2\to\infty$. We decompose $\tilde{Q}(A_1a_0,A_2\overline{a}_0)$ into the following:
\begin{align*}
\tilde{Q}(A_1a_0,A_2\overline{a}_0)&=\tilde{Q}((A_1-I)a_0+a_0,(A_2-I)\overline{a}_0+\overline{a}_0)\\
&=\tilde{Q}(a_0,\overline{a}_0)+\tilde{Q}((A_1-I)a_0,\overline{a}_0)\\
&\quad+\tilde{Q}(a_0,(A_2-I)\overline{a}_0)+\tilde{Q}((A_1-I)a_0,(A_2-I)\overline{a}_0).
\end{align*}
$y_1^{d_1}$ occurs in the second term and $y_2^{d_2}$ occurs in the third term with \textit{non-zero} (in fact, positive) coefficients by the polarization condition.

The sum $\tilde{Q}(A_1f_1,A_2\overline{a}_0)+\tilde{Q}(A_1f_1,A_2\overline{f}_1)+\tilde{Q}(A_1a_0,A_2\overline{f}_1)$ is a polynomial in $y_2$ with degree $\leq d_2$  whose coefficients decay exponentially as $y_1\to\infty$ because of the term $A_1f_1$. Note that we can freely interchange the positions of $A_1$ and $A_2$ inside $\tilde{Q}$. The counterpart $\tilde{Q}(A_1a_0,A_2\overline{f}_2)+\tilde{Q}(A_1f_2,A_2\overline{f}_2)+\tilde{Q}(A_1f_2,A_2\overline{a}_0)$ is similar. We will denote them by $p_2(y_2)$ and $p_1(y_1)$, respectively. Finally, $\tilde{Q}(A_1a_0,A_2\overline{a}_0)$ is a polynomial in $y_1,y_2$, denoted by $\bp(y_1,y_2)$.
\begin{align}\label{Q}
\tilde{Q}=\bp+p_1+p_2+\bfH_{12}.
\end{align}

Formally, the function $\tilde{Q}$ satisfies {\bf (1)} $\tilde{Q}>0$ and {\bf (2)} The matrix $-\p_i\p_{\bar j}\log\tilde{Q}$ is semi-positive definite. 
\begin{defn}\label{cWP}
We say that a polynomial $\bp(y_1,y_2)$ is a \emph{candidate of Weil--Petersson potentials}, if it satisfies {\bf (1')} $\bp> 0$ and {\bf (2')} the matrix $-\p_i\p_{\bar j}\log\bp$ is semi-positive definite as $y_1,y_2$ large enough.
\end{defn}

\begin{remark}
Since $-\log\tilde{Q}$ is a Weil--Petersson potential, the corresponding polynomial $\bp$ in \eqref{Q} must be a candidate of Weil--Petersson potential. Indeed, the functions $p_1,p_2$ and $\bfH_{12}$ decay exponentially. $\bp$ will dominate them as $y_1,y_2$ large. Hence the positivity of $\tilde{Q}$ implies the positivity of $\bp$. 

For the second condition, by a direct differentiation, we first note that
\begin{align}\label{pert}
(-\p_i\p_{\bar j}\log\tilde{Q})=(-\p_i\p_{\bar j}\log\bp)+E,~~E=\begin{bmatrix}a & b\\ \bar b & c\end{bmatrix}.
\end{align}
$E$ is a hermitian matrix (maybe non-positive definite) whose entries are exponentially decay in $y_1,y_2$ (see the computation below). If $(-\p_i\p_{\bar j}\log\bp)$ is not semi-positive definite, then it will be non semi-positive definite along some algebraic curve $C$. Hence $(-\p_i\p_{\bar j}\log\tilde{Q})$ will also be non semi-positive definite along $C$.
\end{remark}

If $E_1$ is infinite and $E_2$ is finite, the same computation shows that $\bf p$ only depends on $y_1$ and $p_2=0$.

Naively, to compute the distance, we should compute the metric matrix $(-\p_i\p_{\bar j}\log\tilde{Q})$ first and then integrate it over any real curve. The first step can be done explicitly (see the computation in {\bf Section 3.1}). For the second step, in light of \eqref{pert}, we may regard $(-\p_i\p_{\bar j}\log\tilde{Q})$ as a perturbation of the metric $(-\p_i\p_{\bar j}\log\bp)$. In {\bf Section 3.2}, we will prove that the metric induced by $(-\p_i\p_{\bar j}\log\bp)$ have infinite distance. So the problem can be formulated into the following statement: 

{\it Suppose ${\bf p}$ is a polynomial in $y_i$ satisfies {\bf (1')} and {\bf (2')}. Let $M:=(-\p_i\p_{\bar j}\log{\bf p})$ be a metric with infinite distance from any smooth point to $s$ and $E$ be any hermitian matrix whose entries consists of exponential decay functions in $y_i$. If $M+E$ is semi-positive definite, then it also has infinite distance.}

Unfortunately, the conditions on $E$ are still too weak (cf.~{\it Example} 3.10) to imply the conjecture in this case since we do not really characterize $M+E$ when it comes from a potential function. We need a detail study on such a perturbation matrix $E$. 

\subsection{Candidates of Weil--Petersson potentials}  In order to understand the metric, we will classify the dominate polynomial $p$ of all the candidates $\bp$ in this section.

\subsubsection{$E_1$ is infinite and $E_2$ is finite} In this case, $\bp$ depends only on $y_1$. Therefore, the condition {\bf (1'), (2')} are satisfied provided that the leading coefficient of $\bp$ is positive.

\subsubsection{Both $E_1$ and $E_2$ are infinite}  Let's examine the derivatives of the dominant term first. Put
\[
f:={\bf p}(y_1,y_2)+p_1(y_1)+p_2(y_2).
\] 
One computes:
\begin{align}
\begin{split}
&[(\p_{y_1} f)^2-f(\p^2_{y_1} f)]/f^2= (\p_{y_1} {\bf p}(y_1,y_2)+\p_{y_1} p_1(y_1)+\p_{y_1} p_2(y_2))^2/{\bf p}^2\\
&-({\bf p}(y_1,y_2)+p_1(y_1)+p_2(y_2))(\p^2_{y_1}{\bf p}(y_1,y_2)+\p^2_{y_1}p_1(y_1)+\p^2_{y_1} p_2(y_2))/{\bf p}^2\\
&\sim [(\p_{y_1}{\bf p}+\p_{y_1} p_2)^2-{\bf p}(\p^2_{y_1}{\bf p}+\p^2_{y_1} p_2)]/{\bf p}^2\\
&\sim[(\p_{y_1} {\bf p})^2-{\bf p}\p^2_{y_1} {\bf p}]/{\bf p}^2+e^{-y_1}(\text{bounded terms})\label{y11}.
\end{split}
\end{align}
For the other half, note that $\p_{x_1}{\bf p}=\p_{x_1}p_1=0$ and $x_1$ only appears together with $y_1$ on the exponents. A similar computation shows that 
\begin{align}\label{x11}
[(\p_{x_1}f)^2-f(\p^2_{x_1} f)]/f^2\sim &\frac{(\p_{x_1} p_2)^2-{\bf p}\p^2_{x_1}p_2}{{\bf p}^2}
\end{align}
Let $D_i=\deg_{y_i}{\bf p}$. The numerator of \eqref{x11} is a polynomial in $y_2$ of degree (at most) $d_2+D_2\leq 2D_2$. All of its coefficients involve $e^{-y_1}$ factors. Summing together, we have:
\begin{align*}
\frac{|\p_1 f|^2-f(\p_1\p_{\bar 1} f)}{f^2}\sim \frac{(\p_{y_1} {\bf p})^2-{\bf p}\p^2_{y_1} {\bf p}}{{\bf p}^2}+e^{-y_1}(\text{bounded terms}).
\end{align*}
Similarly, we have
\begin{align*}
\frac{|\p_2 f|^2-f(\p_2\p_{\bar 2} f)}{f^2}\sim \frac{(\p_{y_2} {\bf p})^2-{\bf p}\p^2_{y_2} {\bf p}}{{\bf p}^2}+e^{-y_2}(\text{bounded terms}).
\end{align*}
For the off-diagonal terms $[(\p_{1} f)(\p_{\bar 2} f)-f(\p_{1\bar 2}f)]/f^2$, the main term is
\begin{align*}
\frac{(\p_{y_1} {\bf p})(\p_{y_2}{\bf p})-{\bf p}\p_{y_1}\p_{y_2} {\bf p}}{{\bf p}^2}
\end{align*}
provided that it is non-zero. All the other terms in the numerator can be dominated by $e^{-y_1}g_2(y_2)$, $e^{-y_2}g_1(y_1)$ or $Ce^{-(y_1+y_2)}$, where $g_i$ is a polynomial of $y_i$ with degree at most $2D_i-2$ and $C>0$ is a constant. Now we consider the cross term:
\begin{align*}
\frac{(\p_{y_1} {\bf p}+\p_{y_1}p_2)(\p_{x_2}p_1)-{\bf p}\p_{y_1}\p_{x_2}p_1}{{\bf p}^2}.
\end{align*}
The term $(\p_{y_1}p_2)(\p_{x_2}p_1)$ is bounded by $Ce^{-(y_1+y_2)}$ with $C>0$. For the remaining two terms:
\begin{align*}
\frac{(\p_{y_1} {\bf p})(\p_{x_2}p_1)-{\bf p}\p_{y_1}\p_{x_2}p_1}{{\bf p}^2}.
\end{align*}
The degree of $y_1$ in the numerator is at most $2D_1-2$ by a direct computation. Note that we can ignore the $y_2$ terms since the factor $e^{-y_2}$ appears in the coefficients. Similarly, the numerator of the other term; namely, 
\begin{align*}
\frac{(\p_{y_2} {\bf p})(\p_{x_1}p_2)-{\bf p}\p_{y_2}\p_{x_1}p_2}{{\bf p}^2}
\end{align*}
is a polynomial in $y_2$ with degree at most $2D_2-2$. The off-diagonal term $-\p_1\p_{\bar 2}\log f$ is dominated by 
\begin{align}\label{off-dia}
\frac{(\p_{y_1} {\bf p})(\p_{y_2}{\bf p})-{\bf p}\p_{y_1}\p_{y_2} {\bf p}}{{\bf p}^2}+C_1e^{-y_2}\frac{y_1^{2D_1-2}}{{\bf p}^2}+C_2e^{-y_1}\frac{y_2^{2D_2-2}}{{\bf p}^2}+C_3\frac{e^{-(y_1+y_2)}}{{\bf p}^2}.
\end{align}

We associate $\mathbf{p}$ with a convex polygon $\bbR^2_+$ as follow. Suppose $y_1^ay_2^b$ is a monomial with non-zero coefficient in $\mathbf{p}$, then we draw a dot at $(a,b)$. Let $p$ be the dominant polynomial; that is, the polynomial consisting of the monomials in {\bf p} given by all the right-upper dots. The asymptotic behavior of $\bp$ is completely determined by $p$. $y_1N_1+y_2N_2$ defines {\it the same} monodromy weight filtration for any positive $y_1,y_2$ \cite{CK1982}. Hence $p$ must have at most homogeneous degree $d=3$ and $1\leq d_i\leq d\leq 3$. We have the following several cases.

\paragraph{\bf I} $d_1=d_2=1$. Then $1\leq d\leq 2$ simply because $y_1^ay_2^bN_1^aN_2^b=0$ for $a\geq 2$ or $b\geq 2$. 

If $d=1$, then $p=Ay_2+By_1$ with $A,B>0$. The corresponding matrix is
\[
M(p)=\frac{1}{p^2}
\left[\begin{matrix}
A^2 & AB\\
AB & B^2
\end{matrix}\right],
\]

If $d=2$, then $p=Ay_2+By_2y_1+Cy_1$ with $A,C>0$. Then we must have $B>0$ by the positivity of $p$. In fact,
\[
M(p)\sim \frac{1}{p^2}
\left[\begin{matrix}
B^2y_2^2 & AC\\
AC & B^2y_1^2
\end{matrix}\right],
\]

One can check easily that the Hessian matrices associated with these $p$ are semi-positive definite.

\paragraph{\bf II} $d_1=1$, $d_2=2$. Then $d=2,3$. Suppose $d=2$. The dominate polynomial could be $p=Ay_2^2+By_2y_1+Cy_1$ or $p=Ay_2^2+By_1$ with $A,B,C>0$. For the later case, the Hessian of $-\log(p)$ is
\[
\frac{1}{p^2}
\left[\begin{matrix}
B^2 & 2ABy_2\\
2ABy_2 & 2A^2y_2^2
\end{matrix}\right],
\] 
which is not semi-positive definite. 

For the other case, the dominant term of $M(p)$ is
\[
\frac{1}{p^2}
\left[\begin{matrix}
B^2y_2^2 & ABy_2^2\\
ABy_2^2 & 2A^2y_2^2+2ABy_1y_2+B^2y_1^2
\end{matrix}\right],
\] 
which is clearly semi-positive definite. So $p=Ay_2^2+By_1y_2+Cy_1$ is a candidate of our Weil--Petersson potential.

For $d=3$, the only possibility is $p=Ay_2^2+By_2^2y_1+Cy_1$ with $A,B,C>0$. The matrix $M(p)$ is semi-positive in this case. it is also a candidate of the potential.

\paragraph{\bf III} $d_1=1$, $d_2=3$. Then $d=3$. In this case, the dominate polynomials could be $Ay_2^3+By_2^2y_1+Cy_1$, $Ay_2^3+By_2y_1+Cy_1$, or $Ay_2^3+By_1$. The case $p=Ay_2^3+By_2^2y_1+Cy_1$ is similar to the previous case. The associated Hessian matrix is semi-positive definite.

For the last two cases, $M(p)$ are not semi-positive definite. The only possibility in this case is $p=Ay_2^3+By_2^2y_1+Cy_2$.

\paragraph{\bf IV} $d_1=d_2=2$. First we assume that $d=2$. $p=Ay_2^2+By_2y_1+Cy_1^2$ with the condition that $A>0$, $C>0$. Note that $p$ is homogeneous of degree two. If $p$ is positive for $y_i$ large enough, then $p$ must be positive for all $y_i>0$. To compute the lower bound of the eigenvalues of $M(p)$, we may assume $y_1^2+y_2^2=1$ via scalings. We calculate $M(p)$:
\begin{align*}
\frac{1}{p^2}
&\left[\begin{matrix}
2A^2y_2^2+2ABy_2y_1+(B^2-2AC)y_1^2 & ABy_2^2+4ACy_2y_1+BCy_1^2\\
ABy_2^2+4ACy_2y_1+BCy_1^2 & (B^2-2AC)y_2^2+2CBy_2y_1+2C^2y_1^2
\end{matrix}\right]\\
&=\frac{1}{p^2}
\left[\begin{matrix}
2Ap+(B^2-4AC)y_1^2 & Bp-(B^2-4AC)y_2y_1\\
Bp-(B^2-4AC)y_2y_1 & 2Cp+(B^2-4AC)y_2^2
\end{matrix}\right].
\end{align*}
We have 
\[
\det(M(p))=\frac{1}{p^2} (B^2-4AC).
\]
Being semi-positive definite, this forces $B^2-4AC\geq 0$. If $B^2-4AC>0$, then the eigenvalues of $M(p)$ have a positive lower bound on the compact set $K:=\{(y_1,y_2)\in\bbR_{\geq 0}^2:y_1^2+y_2^2=1\}$. Also note that $p$ is a complete square if $B^2-4AC=0$.

If $d=3$, there are three possibilities: $p=Ay_2^2+By_2^2y_1+Cy_1^2$, $Ay_2^2+By_2y_1^2+Cy_1^2$ or $Ay_2^2+By_2^2y_1+Cy_2y_1^2+Dy_1^2$. 

For $p=Ay_2^2+By_2^2y_1+Cy_1^2$ being positive, we have $A,B,C>0$. But then $M(p)$ can't be semi-positive definite. Indeed, $\det M(p)<0$ as $y_1$ large in this case. We can rule out this case. The second case is similar. There is only one possibility: $p=Ay_1^2+By_1^2y_2+Cy_1y_2^2+Dy_2^2$. In this case, we have $B,C>0$. As $y_1$, $y_2$ large, 
\begin{align*}
M(p)=\frac{1}{p^2}
\left[\begin{matrix}
C^2y_1^4+2B^2y_2^2y_1^2+2BCy_2y_1^3 & BCy_2^2y_1^2\\
BCy_2^2y_1^2 & B^2y_2^4+2BCy_2^3y_1+2C^2y_2^2y_1^2
\end{matrix}\right].
\end{align*}
This matrix is indeed positive definite.

\paragraph{\bf V} $d_1=3$, $d_2=2$. There are four candidates for the dominate polynomials, namely, $p=Ay_2^3+By_2^2y_1+Cy_2y_1^2+Dy_1^2$, $Ay_2^3+By_2^2y_1+Cy_1^2$, $Ay_2^3+By_2y_1^2+Cy_1^2$ and $Ay_2^3+By_1^2$. The condition making $M(p)$ semi-positive rules out the later two cases. Indeed, for these two cases, we have
\[
\det M(p)<0.
\]
For $p=Ay_2^3+By_2^2y_1+Cy_1^2$. $\p^2_{y_2}(p)\geq 0$ implies $B<0$. And $\det(M(p))\geq 0$ implies $B\geq 0$ by substituting $y_1={y_2}^{3/2}$. 

We only have to consider the case $p=Ay_2^3+By_2^2y_1+Cy_2y_1^2+Dy_1^2$ with $A,C>0$ and $BD\neq 0$. In this case, we have $M(p)=$
\begin{align*}
\frac{1}{p^2}
\left[\begin{matrix}
\begin{split} (B^2-2AC)y_2^4-2ADy_2^3\\+2BCy_2^3y_1+2C^2y_2^2y_1^2 \end{split} & ABy_2^4+4ACy_2^3y_1+BCy_2^2y_1^2\\
& ABy_2^4+4ACy_2^3y_1+BCy_2^2y_1^2 \begin{split}3A^2y_2^4&+4ABy_2^3y_1+2B^2y_2^2y_1^2\\&+2BCy_2y_1^3+C^2y_1^4\end{split}\\
\end{matrix}\right].
\end{align*}
For $A,D>0$, we must have $B^2-3AC\geq 0$ by looking at the determinant when $y_2 \gg y_1$.

\paragraph{\bf VI} $d_1=3$, $d_2=3$. For this case, $d=3$. Let $p=Ay_1^3+By_1^2y_2+Cy_1y_2^2+Dy_2^3$ with $A,D>0$ and $BC\neq 0$. If one of $B,C=0$, $M(p)$ will not be semi-positive definite. 
Note that $p$ is homogeneous of degree three. We can restrict ourselves on the compact set $K$. We observed that as $y_1\to 1^-$ and $y_2\to 0^+$, 
\begin{align*}
M(p)\to \frac{1}{p^2}
\left[\begin{matrix}
3A^2 & AB\\
AB & B^2-2AC
\end{matrix}\right].
\end{align*}
For this to be semi-positive, we must have $B^2-3AC\geq 0$. Similarly, using $y_1\to 0^+$ and $y_2\to 1^-$, we have $C^2-3BD\geq 0$.
A brute force shows that 
\[
\det(M(p))=\frac{2}{f^2}((B^2-3AC)y_1^2+(BC-9AD)y_1y_2+(C^2-3BD)y_2^2).
\]
Note that it is also required that $\det(M(p))\geq 0$.

We make a further substitution. $p$ is a real polynomial of homogeneous degree 3. It has a linear factor over $\bbR$. Note that $p>0$ on the first quadrant. So the factor must be of the form $ty_1+sy_2$ with $s,t>0$. We decompose
\[
p(y_1,y_2)=(ty_1+sy_2)(ay_1^2+by_1y_2+cy_2^2).
\]
We also have $a,c>0$ and $ay_1^2+by_1y_2+cy_2^2>0$ in the first quadrant. The coefficients are related by $A=at$, $B=tb+sa$, $C=tc+sb$ and $D=sc$.

\rm{VI-(a)}: $B^2-3AC>0$, $C^2-3BD>0$ and $BC>0$. Then $BC-9AD>0$. 

\rm{VI-(b)}: $B^2-3AC>0$, $C^2-3BD>0$ and $BC<0$. We will show that the first eigenvalue on $K$ has a positive lower bound in this case. Note that we have $BC-9AD<0$. For $M(p)$ being semi-positive definite, we have 
\[
(BC-9AD)^2-4(B^2-3AC)(C^2-3BD)\leq 0.
\]
If ``$<$" occurs, we are in the previous situation, i.e., the first eigenvalue is bounded below. Suppose ``$=$" holds. Making a further substitution, we have
\[
-3(b^2-4ac)(as^2-bst+ct^2)^2=0.
\]
In either case, we can write $p=p_1^2p_2$, i.e., $p$ factors completely into linear factors with a square factor. But this forces $BC>0$ and contradicts to our assumption.

\rm{VI-(c)}: $B^2-3AC=0$ and $C^2-3BD>0$. In this case, $BC-9AD>0$, otherwise $\det M(p)$ is not positive. 

\rm{VI-(d)}: $B^2-3AC=0$ and $C^2-3BD=0$. Then $p=q^3$ for some polynomial $q$.

We summarize the results in this section by the following table:
\begin{center}
\begin{tabular}{|c|c|c|c|l|l|}
\hline 
Case & $d_1$ & $d_2$ & $d$ & Dominate polynomial $p$ & Further conditions \\\hline  
(i) &	1 & 1 & 1 & $Ay_2+By_1$ & $A,B>0$    \\\hline
(ii) &	1 & 1 & 2 & $Ay_2+By_2y_1+Cy_1$ & $A,B,C>0$ \\\hline
(iii) &	1 & 2 & 2 & $Ay_2^2+By_2y_1+Cy_1$ & $A,B,C>0$ \\\hline
(iv) &	1 & 2 & 3 & $Ay_2^2+By_2^2y_1+Cy_1$ & $A,B,C>0$ \\\hline
(v) &	1 & 3 & 3 & $Ay_2^3+By_2^2y_1+Cy_1$ & $A,B,C>0$ \\\hline
(vi) & 2 & 2 & 2 & $Ay_2^2+By_2y_1+Cy_1^2$ & $A,C>0$, $B^2-4AC\geq 0$ \\\hline
(vii) & 2 & 2 & 3 & $Ay_2^2+By_2^2y_1+Cy_1^2y_2+Dy_1^2$ & $A,B,C,D>0$ \\\hline
(viii) & 2 & 3 & 3 & $Ay_2^3+By_2^2y_1+Cy_1^2y_2+Dy_1^2$ & $A,C,D>0$,  $B^2-3AC\geq 0$ \\\hline
(ix)  & 3 & 3 & 3 & $Ay_2^3+By_2^2y_1+Cy_1^2y_2+Dy_1^3$ & $A,D>0$, $B^2-3AC\geq 0$,\\ 
	   & & & & &  and $C^2-3BD\geq 0$ \\\hline
\end{tabular}\\
\end{center}

\subsection{Toward the Weil--Petersson distance}
In this section, we will show that the conjecture holds under further assumptions (cf.~Introduction).

\subsubsection{$E_1$ is infinite and $E_2$ is finite} In this paragraph, we will show that the Weil--Petersson distance of the intersection of divisors {\it along the angular slices}; namely, $\Re{z_j}=c_j$ on $\bbH^2\times(\Delta)^{r-2}$ for constants $c_j$'s, are infinite. More precisely, we have

\begin{thm}\label{main3}
Suppose $s\in S$ lies on exactly two boundary divisors, say $E_1$ and $E_2$, with $E_1$ infinite and $E_2$ finite. Then $s$ has infinite Weil--Petersson distance along the angular slices.
\end{thm}
\begin{proof}
For simplicity, we start with the case $r=2$ first. The general case will follow by the same method. In light of \eqref{pert}, if we restrict on the angular slices, only the real part of $b$ contributes to the distance function. We may assume the perturbation matrix $E$ is of the form
\[
\begin{bmatrix}
0 & e^{-ry_2}g(y_1,y_2)\\
e^{-ry_2}g(y_1,y_2) & \star
\end{bmatrix}
\]
and the Weil--Petersson metric (up to a constant) is given by
\[
\begin{bmatrix}
1/y_1^2 & 0\\
0 & 0
\end{bmatrix}+
\begin{bmatrix}
0 & e^{-ry_2}g\\
e^{-ry_2}g & \star
\end{bmatrix}
=\begin{bmatrix}
1/y_1^2 & e^{-ry_2}g\\
e^{-ry_2}g & \star
\end{bmatrix}.
\]
Here the function $g$ is a real convergent series in $1/y_1$ and $e^{-y_2}$ and is bounded at infinity. Completing the square and using the semi-positivity of the metric, then, up to a constant, we have 
\begin{align}\label{fin1}
\begin{split}
\eqref{WPmetric}\geq (\bullet)\frac{|dy_1+e^{-ry_2}gdy_2|}{y_1}&\geq (\bullet)\frac{|dy_1+e^{-ry_2}gdy_2|}{y_1-e^{-ry_2}g/r}\\
&\geq (\bullet)\frac{dy_1+e^{-ry_2}gdy_2}{y_1-e^{-ry_2}g/r}.
\end{split}
\end{align}
Here $(\bullet)$ are positive constants (they may be different, but we shall use the same notation to avoid introducing new notations). 
\[
d(e^{-ry_2}g)=-re^{-ry_2}gdy_2+e^{-ry_2}(\p_1 gdy_1+\p_2g dy_2).
\]
The right hand side of \eqref{fin1} (without the constant) becomes
\begin{align}\label{fin2}
d\log|y_1-e^{-ry_2}g/r|-\frac{e^{-ry_2}(\p_1 gdy_1+\p_2g dy_2)}{y_1-e^{-ry_2}g/r}.
\end{align}
The function $\p_1 g$ can be written as $g_1/y_1^2$ with $g_1$ bounded at infinity. Hence 
\[
\left|\int_\gamma \frac{e^{-ry_2}(\p_1 g)dy_1}{y_1-e^{-ry_2}g/r}\right|<\infty.
\]
Similarly, the integration of the second term along any $\gamma$ is bounded. Therefore, by \eqref{fin2}, the Weil--Petersson distance along the angular slices is infinite since 
\[
\log|y_1-e^{-ry_2}g/r|\to\infty,~~~y_1,y_2\to\infty.
\]
We thus proved the case $r=2$. 

Using the same technique, the general cases will follow similarly: We adapt the notations in {\bf Section 2}. Let $s\in E_1\cap E_2$ with $E_1$ infinite, $E_2$ finite and $s\notin E_j$ for other $j$. Consider the real part of the matrix, $\Re(-\p_i\p_{\bar j}\log\tilde{Q})$, and restrict to the angular slice. Using $\bp\to y_1^{2D_1}$, we have
\[
\Re(-\p_i\p_{\bar j}\log\tilde{Q})\geq (\bullet) \frac{1}{y_1^2}(a_{ij}).
\]
The matrix $(a_{ij})$ has the following property: $(a_{ij})$ is semi-positive definite. $a_{11}$ is a non-zero constant. $a_{12}$ is of the form $e^{-ry_2}b_{12}$, where $b_{12}$ is a convergent power series in $1/y_1$, $e^{-y_2}$ and $y_j$ for $j\geq 3$. For $i\geq 3$, $a_{1i}$ are convergent power series in $1/y_1$, $e^{-y_2}$ and $y_j$ for $j\geq 3$. Completing the square, we have
\[
(dy_1+\sum_{i\neq 1}a_{1i}dy_i)^2-(\sum_{i\neq 1}a_{1i}dy_i)^2+\sum_{i,j\geq 2}a_{ij}dy_i\otimes dy_j.
\]
The sum of the later two terms is nonnegative. Then, restricting to the angular slices, 
\begin{align*}
\eqref{WPmetric}\geq (\bullet)\frac{|dy_1+\sum_{i\neq 1}a_{1i}dy_i|}{y_1}&\geq (\bullet)\frac{|dy_1+\sum_{i\neq 1}a_{1i}dy_i|}{y_1-a_{12}/r-\sum_{i\neq 1,2}a_{1i}}\\
&\geq  (\bullet)\frac{dy_1+\sum_{i\neq 1}a_{1i}dy_i}{y_1-a_{12}/r-\sum_{i\neq 1,2}a_{1i}}
\end{align*}
as we did in \eqref{fin1}. Note that $y_1-a_{12}/r-\sum_{i\neq 1,2}a_{1i}\to \infty$ as $y_1,y_2\to\infty$, $y_i\to 0$ for $i\geq 3$. The integration of the difference
\[
d\log\left|y_1-\frac{a_{12}}{r}-\sum_{i\neq 1}a_{1i}\right|-\frac{dy_1+\sum_{i\neq 1}a_{1i}dy_i}{y_1-a_{12}/r-\sum_{i\neq 1,2}a_{1i}}
\]
along any real curve $\gamma$ in the angular slice is finite. Indeed, we only have to care about $y_1$ and $y_2$ directions since the range of $y_i$'s are bounded for $i\geq 3$. Note that $dy_1$ is canceled out from the subtraction. The other terms are always involved $1/y_1^2$ or $e^{-y_2}$ as \eqref{fin2}. A similar computation shows that they are all finite. 
\end{proof}

\begin{remark}
The total distance; namely, involving $x_i$ direction, is not easy to compute. In general, a semi-positive definite (hermitian) perturbation of an infinite distance metric may have finite distance (cf.~{\it Example} \ref{rem1}). There should be some constrains for these perturbation, at least for those potential $-\log\tilde{Q}$ coming from geometric families.
\end{remark}

\subsubsection{$E_1$, $E_2$ are infinite} In this case, we will show that:

\begin{prop}\label{prop}
In the case of Calabi--Yau threefolds, suppose $s\in S$ lies exact on two boundary divisors, say $s\in E_1\cap E_2$, with $E_i$'s being infinite. Then it has infinite distance measured by the dominant term of the candidates of the Weil--Petersson potentials, which is denoted by $g_{cWP}$.
\end{prop}

\begin{proof}
We will analyze all the cases described in {\bf Section 3.1}. Again, for simplicity we assume $r=2$ at this moment.

{\bf Case (i)} $p=Ay_2+By_1$. This is easy since for any $\gamma$ in that slice
\begin{align*}
L_{g_{cWP}}(\gamma)&=\int_\gamma \frac{\sqrt{A^2|dz_2|^2+2AB\Re(dz_1\otimes d{\bar z}_2)+B^2|dz_1|^2}}{Ay_2+By_1}\\
&=\int_\gamma\frac{|Adz_2+Bdz_1|}{Ay_2+By_1}\geq \int_\gamma\frac{Ady_2+Bdy_1}{Ay_2+By_1}=\log(Ay_2+By_1)=\infty.
\end{align*}

{\bf Case (ii)} For the case $p=Ay_2+By_2y_1+Cy_1$, we have
\begin{align*}
L_{g_{cWP}}(\gamma)&\geq (\bullet)\int_\gamma \frac{\sqrt{B^2y_2^2|dz_1|^2+B^2y_1^2|dz_2|^2}}{|p|}\\
&\geq (\bullet)\int_\gamma \frac{By_2|dz_1|+By_1|dz_2|}{\sqrt{2}|p|}\\
&\geq (\bullet)\int_\gamma \frac{y_2dy_1+y_1dy_2}{y_1y_2}=(\bullet)(\log(y_1)+\log(y_2))=\infty,
\end{align*}
for suitable positive constants. The notation $(\bullet)$ is the same as we explained in section {\it 3.2.1}.

{\bf Case (iii)} $p=Ay_2^2+By_2y_1+Cy_1$. It suffices to consider its dominated term $y_2(Ay_2+By_1)$. The metric induced by $p$ is indeed a sum of two semi-positive matrices, which have infinite distance. Therefore, it has infinite distance.

{\bf Case (iv)} This case is similar to {\bf Case (ii)}. We omit the computation here.

{\bf Case (v)} $p=Ay_2^3+By_2^2y_1+Cy_1$, with $A,B,C>0$. The dominated term of $p$ is $Ay_2^3+By_2^2y_1=y_2^2(Ay_2+By_1)$. Hence it has infinite distance as in {\bf Case (iii)}.

{\bf Case (vi)} In this case, if $d=2$, we have either $B^2-4AC>0$ or $B^2-4AC=0$. We deal with the case $B^2-4AC>0$ first. Let
\[
K:=\{(y_1,y_2)\in\mathbb{R}^2:y_i\geq 0,~y_1^2+y_2^2=1\}.
\]
The corresponding matrix has positive eigenvalues on the compact set $K$. Let $\lambda$ be its minimum value on $K$. Let $R(y_1,y_2)=\sqrt{y_1^2+y_2^2}$ and $\eta_j=y_j/R$. Then $p=R^2(A\eta_2^2+B\eta_2\eta_1+C\eta_1^2)$. Note that $A\eta_2^2+B\eta_2\eta_1+C\eta_1^2$ is positive on $K$. Then, for any $\gamma$, 
\[
L_{g_{cWP}}(\gamma)\geq (\bullet) \int_\gamma \frac{dy_1+dy_2}{\sqrt{y_1^2+y_2^2}}\geq (\bullet)\int_\gamma \frac{dy_1+dy_2}{y_1+y_2}=(\bullet)\log(y_1+y_2)=\infty.
\]
If $B^2=4AC$, then $p=q^2$. Then it has infinite distance as we discuss above.
\paragraph{\bf Case (vii)} $p=Ay_2^2+By_2^2y_1+Cy_2y_1^2+Dy_1^2$ with $B,C>0$. Its dominated term is $By_2^2y_1+Cy_2y_1^2=y_2y_1(By_2+Cy_1)$. The metric $g_{cWP}$ decomposes into three semi-positive definite matrices and all of them have infinite distance.

{\bf Case (viii)} $p=Ay_2^3+By_2^2y_1+Cy_2y_1^2+Dy_1^2$ with $A,C,D>0$, $B\neq 0$ and $B^2-3AC\geq 0$. The dominated term is $y_2(Ay_2^2+By_2y_1+Cy_1^2)$. The metric matrix is 
\begin{align*}
\frac{1}{p^2}
\left[\begin{matrix}
\begin{split} (B^2-2AC)y_2^4\\+2BCy_2^3y_1+2C^2y_2^2y_1^2 \end{split} & ABy_2^4+4ACy_2^3y_1+BCy_2^2y_1^2\\
& ABy_2^4+4ACy_2^3y_1+BCy_2^2y_1^2 \begin{split}3A^2y_2^4&+4ABy_2^3y_1+2B^2y_2^2y_1^2\\&+2BCy_2y_1^3+C^2y_1^4\end{split}\\
\end{matrix}\right].
\end{align*}
A direct computation gives
\[
\det(M(p))=\frac{2((B^2-3AC)y_2^2+BCy_2y_1+C^2y_1^2)}{(Ay_2^3+By_2^2y_1+Cy_2y_1^2)^2}.
\]
If $B^2-3AC>0$, then as we did in {\bf Case (vi)}, its first eigenvalue has a positive lower bound on $K$. A parallel argument shows that it has infinite distance. 

If $B^2-3AC=0$, then $B>0$. $4ACy_2^3y_1$ and $BCy_2^2y_1^2$ can be dominated by the diagonal terms. In fact,
\[
2\sqrt{3AC}Cy_2^3y_1+4A\sqrt{3AC}y_2^3y_1\geq 4\sqrt{6}ACy_2^3y_1>8ACy_2^3y_1.
\]
It suffices to show that the matrix
\begin{align*}
\frac{1}{p^2}
\left[\begin{matrix}
\begin{split} ACy_2^4+\eta BCy_2^3y_1+C^2y_2^2y_1^2 \end{split} & ABy_2^4\\
& ABy_2^4\begin{split}3A^2y_2^4&+\eta ABy_2^3y_1+B^2y_2^2y_1^2\\&+2BCy_2y_1^3+C^2y_1^4\end{split}\\
\end{matrix}\right].
\end{align*}
gives an infinite distance on any curve. Here $\eta>0$ is a constant. The key point here is that when we compute the distance,  the off-diagonal terms $2ABy_2^4\Re(dz_1\otimes d{\bar z}_2)$ with the terms $ACy_2^4|dz_1|^2$, $3A^3y_2^4|dz_2|^2$ in diagonal fit together and equal to 
\[
|\sqrt{AC}y_2^2dz_1+\sqrt{3}Ay_2^2dz_2|^2
\]
and all remaining terms are positive. Therefore, 
\begin{align*}
L_{g_{cWP}}(\gamma)&\geq (\bullet) \int_\gamma \frac{(2Ay_2+By_1)dy_2+(2Cy_1+By_2)dy_1}{Ay_2^2+By_1y_2+Cy_1^2}+\frac{dy_2}{y_2}\\
&\geq (\bullet)\log|Ay_2^2+By_1y_2+Cy_1^2|+(\bullet)\log|y_2|=\infty.
\end{align*}

{\bf Case (ix)} The first eigenvalue in the cases \rm{(a),(b)} in section 3.1 have positive lower bound. Hence they all have infinite distance. \rm{(c)} is similar to the second part of {\bf Case (viii)}. The same proof shows that it has infinite distance, too. For the last case, we have $p=q^3$ for some homogeneous degree one polynomial $q$. Then $\p_i\p_j \log p = 3\p_i\p_j \log q$ and it is reduced to {\bf Case (i)}.

Note that for a higher dimensional base $S$, the dominant term of the candidate of the Weil--Petersson potential is the same as the one for two parameter families. With these discussions, this completes the proof of {\bf Proposotion \ref{prop}}.
\end{proof}
\begin{remark}
The statement of {\bf Proposition} \ref{prop} holds true trivially in the situation that $s\in\cap_i E_i$ with $E_1$ infinite and $E_i$ finite for $i\geq 2$. In this case, the dominant term of the candidates of Weil--Petersson potential is just a polynomial in one variable with positive leading term. Therefore,
\[
L_{g_{cWP}}(\gamma)= (\bullet)\int_\gamma \frac{dy_1}{y_1}=\infty.
\]
\end{remark}

\subsubsection{Further Results} For two parameter families, we can go further in some cases. For instance, in {\bf Case (ii)}, the diagonal of the metric matrix is large enough to dominate {\it all} the off-diagonal terms. To be precise, we observe that the metric matrix is
\begin{align}\label{case2}
\frac{1}{p^2}
\left[\begin{matrix}
B^2y_2^2 & AC\\
AC & B^2y_1^2
\end{matrix}\right].
\end{align}
Also, we have the inequalities:
\begin{align*}
&-2AC\left|dz_1\otimes d{\bar z}_2\right|\geq-\left(A^2|dz_1|^2+C^2|dz_2|^2\right)\\
&-C_1\left|e^{-y_2}\frac{1}{p^2}dz_1\otimes d{\bar z}_2\right|\geq -\frac{C_1}{2}\left(\frac{e^{-y_2}}{p^2}|dz_1|^2+\frac{e^{-y_2}}{p^2}|dz_2|^2\right)\\
&-C_2\left|e^{-y_1}\frac{1}{p^2}dz_1\otimes d{\bar z}_2\right|\geq -\frac{C_2}{2}\left(\frac{e^{-y_1}}{p^2}|dz_1|^2+\frac{e^{-y_1}}{p^2}|dz_2|^2\right)\\
&-C_3\left|\frac{e^{-(y_1+y_2)}}{p^2}dz_1\otimes d{\bar z}_2\right|\geq -\frac{C_3}{2}\left(\frac{e^{-(y_1+y_2)}}{p^2}|dz_1|^2+\frac{e^{-(y_1+y_2)}}{p^2}|dz_2|^2\right).
\end{align*}
As $y_2$ large, we can use the diagonal of \eqref{case2} to dominate all the right hand sides of the inequalities simultaneously. Thus
\begin{align*}
L_{g_{WP}}(\gamma)&\geq (\bullet)\int_\gamma \frac{By_2|dz_1|+By_1|dz_2|}{By_2y_1}\\
&\geq (\bullet)\int_\gamma \frac{By_2|dy_1|+By_1|dy_2|}{By_2y_1}\\
&\geq (\bullet)\int_\gamma \frac{By_2dy_1+By_1dy_2}{By_2y_1}=(\bullet)\log(y_1y_2)=\infty.
\end{align*}

{\bf Case (iii)-(v)} and {\bf (vii)} are similar. For the {\bf Case (vi)}, if $B^2-4AC>0$, we can also use the diagonal terms to dominate the off-diagonal terms because the first eigenvalue is strictly positive on $K$. Let $\lambda>0$ be the minimal eigenvalue of $(-\p_i\p_j\log\bp)$. Then 
\[
(-\p_i\p_j\log\bp)\geq \frac{1}{R^2}\begin{bmatrix}\lambda & 0\\ 0 & \lambda\end{bmatrix},~~R=\sqrt{y_1^2+y_2^2}.
\]
Using the same technique as above, we can dominate the last three terms in \eqref{off-dia} by $\lambda/R^2$ as $y_i$ large. Hence
\[
(-\p_i\p_{\bar j}\log\tilde{Q})\geq \frac{1}{2R^2}\begin{bmatrix}\lambda & 0\\ 0 & \lambda\end{bmatrix}.
\]
A direct computation shows that it has infinite distance. {\bf Case (viii)-(ix)} are similar if all ``$\geq$" are replaced by ``$>$". As a corollary, we have

\begin{cor}\label{cor2}
Let $E_1$, $E_2$ be infinite divisors and $N_1$, $N_2$ be the logarithmic part of the monodromy operators. If $\{D_1,D_2\}=\{1,2\}$ or $\{1,3\}$, then the intersection point has infinite Weil--Petersson distance.

In fact, if the dominant term of the Weil--Petersson potential is of the form described in ${\bf Case (ii)-(ix)}$ with all ``$\geq$" replaced by ``$>$", then it is at infinite distance. 
\end{cor}

\begin{remark}
For the higher dimensional base $S$, a more detailed analysis of the bounded terms appearing in \eqref{y11} is needed. More precisely, for $j\geq 3$, 
\begin{align*}
-\p_1\p_{\bar j}\log\tilde{Q}=\frac{(\p_1\tilde{Q})(\p_{\bar j}\tilde{Q})-\tilde{Q}(\p_1\p_{\bar j}\tilde{Q})}{\tilde{Q}^2}.
\end{align*}
We do not have a good control on the growing/decaying rate of $y_2$-terms in $-\p_1\p_{\bar j}\log\tilde{Q}$ at this moment. Such constrains are closed related to the degeneration of mixed Hodge structures on further intersections of boundary divisors.  Results on degeneration of Hodge structures \cite{S1973}, \cite{K1981}, \cite{K2010} and \cite{CK1982} connect the Hodge structures (the open part of $S$) and limiting mixed Hodge structures (the $1$-boundary part of $S$) and yields the {\bf Observation} \ref{obv}.
\end{remark}

\begin{remark}\label{diam}
As we pointed out in introduction, the validity of {\bf Conjecture \ref{conj}} implies the diameter boundedness of the family $\mathfrak{X}/S$: If $\mathfrak{X}_s$ is at finite distance along a \emph{real curve} $\gamma$, the conjecture says that $N_iF^n_\infty(s)=0$ for all $i$. Then $\mathfrak{X}_s$ has finite distance along any holomorphic curve and hence $\mathfrak{X}/S$ has a uniform diameter bound \emph{over} $S$, in particular, over $\gamma$. 
\end{remark}

\begin{example}\label{rem1}
In general, we may regard the matrix form of \eqref{WPmetric} as a Hermitian perturbation of a given matrix \eqref{pert}. Consider 
\[
M=\frac{1}{y_1^2}\left[\begin{matrix}
1 & 0 \\
0 & 0 
\end{matrix}\right]~ \text{by a hermitian matrix}~E=\left[\begin{matrix}
0 & ie^{-y_2} \\
-ie^{-y_2} & e^{-2y_2} 
\end{matrix}\right]
\]
In this case, the Weil--Petersson distance is asymptotic to the integration of
\[
\frac{1}{y_1^2}\sqrt{|x_1'-e^{-y_2}y_2'|^2+|y_1'+e^{-y_2}x_2'|^2}.
\]
Consider the curve $\gamma$: $t\mapsto (C,t,-e^t,t)$. The second becomes zero since $dt-e^{-t}e^tdt=0$. Thus the distance is given by
\[
\frac{1}{y_1^2}\sqrt{|x_1'-e^{-y_2}y_2'|^2}
\]
and we have 
\[
\int_\gamma \frac{e^{-t}dt}{t}<\infty.
\]
Therefore, in general, we need more constrains on the potential function $\log\tilde{Q}$; that is, a more detailed information on the variation of mixed Hodge structures along the boundaries and its further degenerations.
\end{example}

\subsubsection{Applications on two parameter families}
Explicit examples on two parameter family of Calabi--Yau threefolds were considered by \cite{COFKM1994} and \cite{HKTY1995}. The Weil--Petersson distance can be described using the results in this paper. As an application, we focus on the example $\mathbb{P}(1,1,2,2,2)[8]$, degree $8$ hypersurfaces in the weighted projective space.

We follow the notations in \cite{COFKM1994}. The $1$-boundary parts are completely determined by {\bf Theorem} \ref{thm02} (cf.~Figure 1 in \cite{COFKM1994}). The divisors $C_1$ and $C_{con}$ (cf.~Figure 1 and {\bf Section 4} in \cite{COFKM1994}) are finite divisors. Indeed, the corresponding logarithmic part of the monodromies $N_1$ and $N_{con}$ are of rank one. So the intersection $C_1\cap C_{con}$ is at finite distance place. 

For the other intersection points, we look at the blown up moduli (cf.~Figure 4 in \cite{COFKM1994}). $D_{(0,-1)}$ and $D_{(-1,-1)}$ are divisors with nilpotent indices $4$. The points on them are maximal nilpotency points because they are Calabi--Yau threefolds. They are infinite divisors. Hence $C_\infty=D_{(1,0)}$ is an infinite divisor since the blow-ups of the intersection of finite divisors are still finite.

The intersection $D_{(0,-1)}\cap C_\infty$ is at infinite distance, since it is of the type $(3,1)$ in {\bf Corollary \ref{cor2}}. The intersection $D_{(0,-1)}\cap C_1$ is at infinite distance place \emph{along the angular slices}. The intersection $D_{(0,-1)}\cap D_{(-1,-1)}$ has infinite distance \emph{measured by the dominant term of the candidates of the Weil--Petersson potentials}. We point out that the logarithmic part of the monodromies around $D_{(0,-1)}$ and $D_{(-1,-1)}$ are proportional. It is the {\bf Case (ix)} with all equalities hold.

\end{document}